\documentclass[11pt]{amsart}
\usepackage{amsmath, amsthm, mathabx,amscd, amsfonts, amssymb, hyperref, mathrsfs, textcomp, epsfig, a4wide, graphicx, IEEEtrantools, tikz, verbatim, xypic,enumerate}

\usetikzlibrary{matrix}

\newtheorem{thm}{Theorem}

\newtheorem{prop}{Proposition}
\newtheorem{lemma}{Lemma}
\newtheorem{cor}{Corollary}
\theoremstyle{definition}

\newtheorem{rmk}{Remark}

\theoremstyle{remark}
\newtheorem{remark}{Remark}

     \RequirePackage{rotating}                   
    \def\HSt{%
       \setbox0=\hbox{$\widehat{\mathit{HS}}$}
       \setbox1=\hbox{$\mathit{HS}$}
       \dimen0=1.1\ht0
       \advance\dimen0 by 1.17\ht1
       \smash{\mskip2mu\raise\dimen0\rlap{%
          \begin{turn}{180}
              {$\widehat{\phantom{\mathit{HS}}}$}
           \end{turn}} \mskip-2mu    
                \mathit{HS}
    }{\vphantom{\widehat{\mathit{HS}}}}{}}

     \RequirePackage{rotating}                   
    \def\HMt{%
       \setbox0=\hbox{$\widehat{\mathit{HM}}$}
       \setbox1=\hbox{$\mathit{HM}$}
       \dimen0=1.1\ht0
       \advance\dimen0 by 1.17\ht1
       \smash{\mskip2mu\raise\dimen0\rlap{%
          \begin{turn}{180}
              {$\widehat{\phantom{\mathit{HM}}}$}
           \end{turn}} \mskip-2mu    
                \mathit{HM}
    }{\vphantom{\widehat{\mathit{HM}}}}{}}
    
    \newcommand{\HMb}{\overline{\mathit{HM}}}
\newcommand{\HMf}{\widehat{\mathit{HM}}}
    \newcommand{\HSb}{\overline{\mathit{HS}}}
        \newcommand{\HMa}{\overrightarrow{\mathit{HM}}}
\newcommand{\HSf}{\widehat{\mathit{HS}}}
\newcommand{\HMr}{{\mathit{HM}}}
\newcommand{\spin}{\mathfrak{s}}
\newcommand{\ztwo}{\mathbb{F}}
\newcommand{\Pin}{\mathrm{Pin}(2)}
\newcommand{\Rin}{\mathcal{R}}

\newcommand{\s}{\mathbf{s}}
\newcommand{\x}{\mathbf{x}}

\newcommand{\con}{\mathbf{c}}
\newcommand{\tow}{\mathcal{T}^+}

\begin{document}

\title{Indefinite Stein fillings and $\mathrm{PIN}(2)$-monopole Floer homology}

\author{Francesco Lin}
\address{Columbia University} 
\email{flin@math.columbia.edu}

\begin{abstract}Given a spin$^c$ rational homology sphere $(Y,\spin)$ with $\spin$ self-conjugate and for which the reduced monopole Floer homology $\HMr_{\bullet}(Y,\spin)$ has rank one, we provide obstructions to the intersection forms of its Stein fillings which are not negative definite. The proof of this result (and of its natural generalizations we discuss) uses $\mathrm{Pin}(2)$-monopole Floer homology.\end{abstract}

\maketitle

In the past twenty years, the topology of Stein fillings of contact three-manifolds has been a central topic of research at the intersection of symplectic geometry, gauge theory and mapping class groups (see \cite{OzSti}, \cite{CE} for introductions to the topic). One of the main goals is to classify the Stein fillings of a given contact three-manifold $(Y,\xi)$, the prototypical example given by Eliashberg's celebrated result that any Stein filling of $(S^3,\xi_{std})$ is diffeomorphic to $D^4$ \cite{Eli}. Following that, the classification has been carried over in several other cases of interest, see for example \cite{McD}, \cite{Lis}, \cite{Wen}.
\par
In this paper, we focus on the less ambitious yet very challenging task of understanding the algebraic topology (e.g. the intersection form) of the Stein fillings of a given contact three-manifold $(Y,\xi)$. While many classification results rely on the characterization of rational symplectic $4$-manifolds \cite{McD}, and therefore are mostly effective when dealing with negative definite Stein fillings, we will focus here on the case of indefinite fillings. We will discuss an application toward this goal of monopole Floer homology \cite{KM}, and in particular its $\Pin$-equivariant version defined in \cite{Lin}. The latter is a Floer theoretic analogue of the invariants introduced in \cite{Man}, and can be used to provide an alternative (but formally identical) disproof of the Triangulation Conjecture.
\par Throughout the paper, all Floer homology groups will be taken with coefficients in $\ztwo$, the field with two elements, and we will focus for simplicity on the case of a rational homology sphere $Y$. To give some context, recall that for a given spin$^c$ structure $\spin$, if the reduced Floer homology group $\HMr_{\bullet}(Y,\spin)$ vanishes then all Stein fillings of contact structures $(Y,\xi)$ for which $\spin_{\xi}=\spin$ have $b_2^+=0$ \cite{OS}, \cite{Ech}.
\par
We focus first on the next simplest case in which $\HMr_{\bullet}(Y,\spin)=\ztwo$ has rank one, under the additional assumption that $\spin$ is self-conjugate, i.e. $\spin=\bar{\spin}$ (or, equivalently, $\spin$ is induced by a spin structure). This in turn implies (see Section \ref{review} below) that either 
\begin{align*}
\HMt_{\bullet}(Y,\spin)&=\tow_{-2h(Y)}\oplus \ztwo_{-2h(Y)}\\
\HMt_{\bullet}(Y,\spin)&=\tow_{-2h(Y)}\oplus \ztwo_{-2h(Y)-1}.
\end{align*}
where $h(Y,\spin)$ is the Fr\o yshov invariant of $Y$. Here $\tow$ is the $\ztwo[[U]]$-module $\ztwo[[U,U^{-1}]/U\cdot\ztwo[[U]]$ (to which we refer as the tower), and the indices determine the grading of the lowest degree element. We call these two cases \textit{Type} $I$ and \textit{Type} $II$ respectively. By Poincar\'e duality, if $(Y,\spin)$ has Type $I$ then $(-Y,\spin)$ has Type $II$, and viceversa.
\begin{rmk}
For the reader more accustomed with Heegaard Floer homology, let us point out that under the isomorphism between the theories (\cite{Tau},\cite{CGH} or \cite{KLT} and subsequent papers) we have $\HMr_{\bullet}(Y,\spin)\equiv HF^{red}_*(Y,\spin)$ and $d(Y,\spin)=-2h(Y,\spin)$.
\end{rmk}
We have the following.
\begin{thm}\label{main}
Let $(Y,\xi)$ with $\spin_{\xi}=\spin$ self-conjugate, and suppose $\HMr_{\bullet}(Y,\spin)=\ztwo$.
\begin{itemize}
\item if $(Y,\spin)$ has Type $I$, and $(W,\omega)$ is a Stein filling of $(Y,\xi)$ which is \textbf{not} negative definite, then $W$ has even intersection form and  $b_2^+=1$ and $b_2^-=-8h(Y,\spin)+9$.
\item  if $(Y,\spin)$ has Type $II$, and $(W,\omega)$ is a Stein filling of $(Y,\xi)$ which is \textbf{not} negative definite, then $W$ has even intersection form and $b_2^+=2$ and $b_2^-=-8h(Y,\spin)+10$.
\end{itemize}

\end{thm}
There are plenty of rational homology spheres satisfying the hypotheses of Theorem \ref{main}, including among the others many Seifert fibered spaces. Indeed, our result should be thought as a generalization (via the classification of unimodular lattices \cite{GS}) of the results in \cite{Sti} that classify the intersection forms of Stein fillings of the Brieskorn sphere $\pm\Sigma(2,3,11)$; this has Type $II$ (resp. Type $I$) and $h=\mp1$. In particular, it is shown in \cite{Sti} that a non-negative definite filling of $-\Sigma(2,3,11)$ has intersection form $H$, while a non-negative definite filling of $\Sigma(2,3,11)$ has intersection form $2H\oplus 2E_8$. While the key input of \cite{Sti} is the existence of a nice embedding of $-\Sigma(2,3,11)$ inside the $K3$ surface as the boundary of the nucleus $N(2)$, our proof is purely Floer theoretic, hence it works when there are no explicit topological inputs.
\begin{remark}
The canonical contact structure on $\Sigma(2,3,11)$ also admits a negative definite Stein filling, namely the (deformation of the) resolution of the singularity $\{x^2+y^3+z^{11}=0\}\subset \mathbb{C}^3$. An example of indefinite Stein filling is given by the Milnor fiber of this singularity.
\end{remark}
The Brieskorn sphere with opposite orientation $-\Sigma(2,3,7)$ has Type $I$ with $h=0$; therefore all indefinite fillings have intersection form $H\oplus E_8$. We can generalize this example to all integral surgeries on the figure-eight knot $M(n)$ (where $-\Sigma(2,3,7)=M(-1)$), all of which have $\HMr_{\bullet}(M(n))=\ztwo$ supported in a self-conjugate spin$^c$ structure \cite{OSab}. These manifolds are hyperbolic for $|n|\geq5$ \cite{Mar}. Contact structures on these manifolds, and their fillability properties, have been thoroughly studied in \cite{CK}. For example, the authors construct for $n=-1,\dots,-9$ a Stein fillable contact structure $\xi$ on $M(n)$ for which $\spin_{\xi}$ is self-conjugate (as it has non-vanishing reduced contact invariant). For negative $n$, $(Y, \spin_{\xi})$ has Type $I$ and $h=-\frac{n+1}{8}$, so that our main result implies that any indefinite Stein filling of $(M(n),\xi)$ has $b_2^+=1$ and $b_2^-=10+n$.
\\
\par
Our first theorem has consequences regarding finiteness questions for Stein fillings. In particular, it was conjectured (Conjecture $12.3.16$ of \cite{OzSti}) that given a contact manifold $(Y,\xi)$, the Euler characteristic of its Stein fillings can only attain finitely many values. While this was later shown to be false \cite{BV}, it is still of great interest to understand for which classes of contact three-manifolds this finiteness condition holds. On our end, we have the following.
\begin{cor}\label{maincor}
Given a contact rational homology sphere $(Y,\xi)$, suppose $\spin_{\xi}=\spin$ is self-conjugate and $\HMr_{\bullet}(Y,\spin)$ has rank one. Then the set of Euler characteristics of Stein fillings of $(Y,\xi)$ is finite.
\end{cor}

Even though the statement of Theorem \ref{main} itself does not involve $\Pin$-monopole Floer homology \cite{Lin}, its proof relies on it. Indeed, $\Pin$-symmetry has been already fruitfully exploited to study the topology of closed symplectic $4$-manifolds \cite{Bau}, \cite{Li}, and our result can be thought as an analogue of those in the case of Stein fillings. The key observation is that the bar version of $\Pin$-monopole Floer homology $\HSb_{\bullet}$, unlike its standard counterpart $\HMb_{\bullet}$, is non-trivial for cobordisms with small positive $b_2^+$ \cite{Lin3}; this in turn is a version for cobordisms of the classical Donaldson Theorems B and C \cite{DK}.
\\
\par
More generally, we can provide obstructions on the intersection form of the indefinite Stein fillings of $(Y,\xi)$ in terms of the interaction with $\Pin$-monopole Floer homology of the contact invariant $\con(\xi)\in\HMt_{\bullet}(-Y,\spin_{\xi})$ defined in \cite{KMOS}; as the precise statements require some additional background and terminology (which we review in Section \ref{review}), we postpone them to Section \ref{continv}. Let us point out that the obstructions we provide are applicable in concrete situations; for example, we will show that for a certain Stein fillable contact structure on the Brieskorn homology sphere $\Sigma(2,8k+3,16k+7)$ constructed in \cite{Kar}, all indefinite Stein fillings have intersection form $2H\oplus E_8$. Notice that, provided $k>1$, $\HMr_{\bullet}(\Sigma(2,8k+3,16k+7))$ has rank strictly larger than $1$.
\\
\par
\textit{Acknowledgements.} The author would like to thank Andr\'as Stipsicz for a helpful conversation on the subject. This work was partially funded by NSF grant DMS-1807242.

\vspace{0.3cm}

\section{$\Pin$-monopole Floer homology}\label{review}
We will review some background on the $\Pin$-monopole Floer homology groups defined in \cite{Lin}. We will assume the reader to be familiar with the formal properties of usual monopole Floer homology \cite{KM}; for a gentle introduction, we refer the reader to \cite{Lin2}. Consider $(Y,\spin)$ a rational homology sphere equipped with a self-conjugate spin$^c$ structure. The Seiberg-Witten equations then admit a $\Pin$-symmetry, where
\begin{equation*}
\Pin=S^1\cup j\cdot S^1\subset \mathbb{H}.
\end{equation*}
In particular, the action of $j$ induces an involution $\jmath$ on the moduli spaces of configurations $\mathcal{B}(Y,\spin)$. One can then exploit this extra symmetry to provide refined Floer theoretic invariants
\begin{equation}\label{longex}
\dots\stackrel{i_*}{\longrightarrow}\HSt_{\bullet}(Y,\spin)\stackrel{j_*}{\longrightarrow} \HSf_{\bullet}(Y,\spin)\stackrel{p_*}{\longrightarrow} \HSb_{\bullet}(Y,\spin)\stackrel{i_*}{\longrightarrow}\dots
\end{equation}
which are read respectively \textit{H-S-to}, \textit{H-S-from} and \textit{H-S-bar}. These are also absolutely graded modules over the graded ring 
\begin{equation*}
\Rin=\ztwo[[V]][Q]/(Q^3)
\end{equation*}
where $V$ and $Q$ have degree respectively $-4$ and $-1$. We will be mostly interested in the $\HSf_{\bullet}(Y,\spin)$ version. This invariant fits in the Gysin exact sequence

\begin{center}
\begin{tikzpicture}
\matrix (m) [matrix of math nodes,row sep=1em,column sep=0.5em,minimum width=2em]
  {
 \HSf_{\bullet}(Y,\spin)  && \HSf_{\bullet}(Y,\spin)\\
  &\HMf_{\bullet}(Y,\spin)\\};
  \path[-stealth]
  (m-1-1) edge node [above]{$\cdot Q$} (m-1-3)
   (m-2-2) edge node [left]{$\pi_*$}(m-1-1)
  (m-1-3) edge node [right]{$\iota$} (m-2-2)  
  ;
\end{tikzpicture}
\end{center}
where $\pi_*$ and $\iota$ have degree $0$. This is an exact sequence of $\Rin$-modules, where the $\Rin$-action on $\HMf_{\bullet}(Y,\spin)$ is obtained by sending $Q$ to $0$ and $V$ to $U^2$. In the case $Y=S^3$, we have $\HSf_{\bullet}(S^3)\equiv \Rin\langle-1\rangle$ and this fits together with $\HMf_{\bullet}(S^3)=\ztwo[[U]]\langle-1\rangle$ in gradings between $-4k$ and $-4k-3$ for $k\geq0$ as follows:
\begin{center}
\begin{tikzpicture}
\matrix (m) [matrix of math nodes,row sep=1em,column sep=1em,minimum width=2em]
  {\cdot&\cdot&\cdot\\
 \ztwo  &\ztwo&\ztwo\\
  \ztwo&\cdot&\ztwo\\
  \ztwo & \ztwo&\ztwo\\};
  \path[-stealth]
  (m-2-1) edge node{} (m-2-2)
   (m-2-3) edge node{} (m-3-1)
    (m-3-3) edge node{} (m-4-1)
     (m-4-2) edge node{} (m-4-3)
  ;
\end{tikzpicture}
\end{center}
Here the three columns represent (left to right) $\HSf_{\bullet}$, $\HMf_{\bullet}$ and $\HSf_{\bullet}$ respectively, and the horizontal maps are $\iota$ and $\pi_*$.
\par
In general, we always have for a rational homology sphere that, up to grading shift,
\begin{equation}\label{hsb}
\HSb_{\bullet}(Y,\spin)\equiv\ztwo[[V,V^{-1}][Q]/(Q^3),
\end{equation}
and that $p_*$ is an isomorphism in degrees low enough and zero in degrees high enough.
\\
\par
Using the Gysin exact sequence, one can show the following \cite{HKL}.
\begin{thm}\label{rankone}[Corollary $2$ of \cite{HKL}]
Suppose $\spin$ is self-conjugate, and $\HMr_{\bullet}(Y,\spin)=1$. Then the extra $\ztwo$ summand is in degree either $-2h(Y,\spin)$ or $-2h(Y,\spin)-1$.
\end{thm}
Accordingly, we have that one of the two alternatives
\begin{align*}
\HMf_{\bullet}(Y,\spin)&=\ztwo[[U]]_{-2h(Y,\spin)-1}\oplus \ztwo_{-2h(Y,\spin)}\\
\HMf_{\bullet}(Y,\spin)&=\ztwo[[U]]_{-2h(Y,\spin)-1}\oplus \ztwo_{-2h(Y,\spin)-1}
\end{align*}
holds, corresponding to our definitions of Type $I$ and $II$. When $(Y,\spin)$ has Type $I$, we have the identification as $\ztwo[[V]]$-modules
\begin{equation*}
\HSf_{\bullet}(Y,\spin)=\ztwo[[V]]\langle-2h-3\rangle\oplus\ztwo[[V]]\langle-2h\rangle\oplus \ztwo[[V]]\langle-2h-1\rangle
\end{equation*}
and the action of $Q$ sends each column injectively into the next one. The corresponding high degree part of the Gysin exact sequence looks like
\begin{center}
\begin{tikzpicture}
\matrix (m) [matrix of math nodes,row sep=1em,column sep=1em,minimum width=2em]
  {
 \ztwo  &\ztwo&\ztwo\\
  \ztwo&\ztwo&\ztwo\\
  \cdot & \cdot&\cdot\\
  \ztwo  &\ztwo&\ztwo\\
  \ztwo&\cdot&\ztwo\\
  \ztwo & \ztwo&\ztwo\\};
  \path[-stealth]
  (m-1-1) edge node{} (m-1-2)
   (m-1-3) edge node{} (m-2-1)
    (m-2-2) edge node{} (m-2-3)
  (m-4-1) edge node{} (m-4-2) 
   (m-6-2) edge node{} (m-6-3)
    (m-4-3) edge node{} (m-5-1)   
     (m-5-3) edge node{} (m-6-1) 
  ;
\end{tikzpicture}
\end{center}
When $(Y,\spin)$ has Type $II$, we have the identification as $\ztwo[[V]]$-modules
\begin{equation*}
\HSf_{\bullet}(Y,\spin)=\ztwo[[V]]\langle-2h-3\rangle\oplus\ztwo[[V]]\langle-2h-4\rangle\oplus \ztwo[[V]]\langle-2h-1\rangle
\end{equation*}
and the action of $Q$ sends each column injectively into the next one. The corresponding high degree part of the Gysin exact sequence looks like
\begin{center}
\begin{tikzpicture}
\matrix (m) [matrix of math nodes,row sep=1em,column sep=1em,minimum width=2em]
  {
 \ztwo  &\ztwo\oplus \ztwo&\ztwo\\
  \cdot&\cdot&\cdot\\
  \ztwo  &\ztwo&\ztwo\\
  \ztwo&\cdot&\ztwo\\
  \ztwo & \ztwo&\ztwo\\};
  \path[-stealth]
  (m-1-1) edge node{} (m-1-2) 
  (m-1-2) edge node{} (m-1-3) 
  (m-1-2.south east) edge[bend left, dotted] node{} (m-3-2)
   (m-3-2) edge[bend left, dotted] node{} (m-5-2)
   (m-3-1) edge node{} (m-3-2) 
   (m-5-2) edge node{} (m-5-3) 
   (m-3-3) edge node{} (m-4-1) 
   (m-4-3) edge node{} (m-5-1) 
   ;
\end{tikzpicture}
\end{center}
Here, in the middle column we have highlighted the $U$-action on the $\ztwo[[U]]$-summand using dotted arrows. We record the following observations.
\begin{lemma}\label{observations}
Suppose $\HMr_{\bullet}(Y,\spin)=\ztwo$. Then the map $p_*$ in $\Pin$-monopole Floer homology is injective. Furthermore, the generator of $\HMr_{\bullet}(Y,\spin)$ is in the image of $\iota$.
\end{lemma}
A cobordism $(W,\spin)$ with $\spin=\bar{\spin}$ induces a map on the corresponding $\Pin$-monopole Floer homology groups; this fits in a natural commutative diagram with the Gysin exact sequence and the map induced on the usual Floer homology groups. The key observation underlying our results is the following.
\begin{prop}\label{smallb}
Given two rational homology spheres $(Y_0,\spin_0)$ and $(Y_1,\spin_1)$ equipped with self-conjugate spin$^c$ structures, and consider a cobordism $(W,\spin)$ with $\spin$ self-conjugate and $b_1=0$ between them. Consider the induced map
\begin{equation*}
\HSb_{\bullet}(W,\spin):\HSb_{\bullet}(Y_0,\spin_0)\rightarrow \HSb_{\bullet}(Y_1,\spin_1),
\end{equation*}
where we can identify both groups up to grading shift as in Equation (\ref{hsb}). Then:
\begin{itemize}
\item if $b_2^+=0$, $\HSb_{\bullet}(W,\spin)$ acts as multiplication by $V^k$ for some $k$ (so it is an isomorphism);
\item if $b_2^+=1$, $\HSb_{\bullet}(W,\spin)$ acts as multiplication by $QV^k$ for some $k$;
\item if $b_2^+=2$, $\HSb_{\bullet}(W,\spin)$ acts as multiplication by $Q^2V^k$ for some $k$;
\item if $b_2^+\geq3$, $\HSb_{\bullet}(W,\spin)$ vanishes.
\end{itemize}
In all cases, the value of $k$ is determined by the grading shift, which is given by the formula $\frac{1}{4}(5b_2^+-b_2^-)$.
\end{prop}
Graphically, the cases $b_2^+=0,1,2$ look like
\begin{center}
\begin{tikzpicture}
\matrix (m) [matrix of math nodes,row sep=1em,column sep=1em,minimum width=2em]
  {
 \ztwo  &\ztwo&\quad&\ztwo&\ztwo&\quad&\ztwo&\ztwo\\
  \ztwo&\ztwo&\quad&\ztwo&\ztwo&\quad&\ztwo&\ztwo\\
  \ztwo  &\ztwo&\quad&\ztwo&\ztwo&\quad&\ztwo&\ztwo\\};
  \path[-stealth]
  (m-1-1) edge node{} (m-1-2) 
    (m-2-1) edge node{} (m-2-2) 
 (m-3-1) edge node{} (m-3-2) 
 (m-1-4) edge node{} (m-2-5) 
  (m-2-4) edge node{} (m-3-5) 
 (m-1-7) edge node{} (m-3-8)
 
   ;
\end{tikzpicture}
\end{center}
after a suitable grading shift.
\begin{proof}
The first three bullets were proved in \cite{Lin} and \cite{Lin3}. Let us recall the basic idea behind their proof. In the case $b_2^+=0$, the relevant reducible moduli spaces are transversely cut out and consist of a copy of $\mathbb{C}P^1$. In the case $b_2^+=1$, before adding perturbations in the blow-up, the moduli space of solutions is again $\mathbb{C}P^1$ but in this case the linearization has one dimensional cokernel; adding an extra perturbation in the blow-up corresponds to choosing a generic equivariant map
\begin{equation*}
s:\mathbb{C}P^1\rightarrow \mathbb{R}
\end{equation*}
where the action is the antipodal on the source and $x\mapsto -x$ on the target. The zero set $s^{-1}(0)$ is a generator of the one dimensional homology of $\mathbb{C}P^1/\jmath=\mathbb{R}P^2$, which corresponds to the claimed computation of multiplication by $Q$. Notice that a choice of such an equivariant section is equivalent to choosing a section of the tautological line bundle over $\mathbb{R}P^2$.
Similarly, when $b_2^+=2$ we are dealing with equivariant maps $\mathbb{C}P^1\rightarrow \mathbb{R}^2$; here the zero set consists of $4N+2$ points, which again descends to a generator of the zero dimensional homology of $\mathbb{R}P^2$.
\par
Now, if $b_2^+=4k+3$, the map is zero for grading reasons. Let us focus on the case $b_2^+=4k+2$. In this case, before adding perturbations the relevant moduli space of reducible solutions is a copy of $\mathbb{C}P^{2k+1}$, and the cokernel has real dimension $4k+2$. The perturbations in the blow-up correspond to equivariant maps
\begin{equation*}
\mathbb{C}P^{2k+1}\rightarrow \mathbb{R}^{4k+2},
\end{equation*}
where the action on $\mathbb{C}P^{2k+1}$ is given by
\begin{equation*}
[z_1:w_1:\cdots:z_{k+1}:w_{k+1}]\mapsto[-\bar{w}_1:\bar{z}_1:\cdots:-\bar{w}_{k+1}:\bar{z}_{k+1}]
\end{equation*}
and is again $x\mapsto -x$ on the target. We are interested in the zero locus of such map. On the other hand, when $k\geq 1$ any such map is homotopic to a map with no zeroes, because there is an equivariant map
\begin{equation*}
\mathbb{C}P^{2k+1}\rightarrow\mathbb{R}^{3k+3}
\end{equation*}
which is never zero. This is obtained by taking $k+1$-copies of the standard squaring map
\begin{align*}
\mathbb{C}^2&\rightarrow \mathbb{R}^3\\
(\alpha,\beta)&\mapsto\left(\frac{1}{2}(|\alpha|^2-|\beta|^2),\mathrm{Re}(\alpha,\bar{\beta}),\mathrm{Im}(\alpha\bar{\beta})\right),
\end{align*}
restricting to the unit sphere in $\mathbb{C}^{2k+2}$, and modding out by the $S^1$-action. This implies that the induced map vanishes. Finally, the cases $b_2^+=4k$ or $4k+1$ follow from this by stabilizing with an equivariant map with values in $\mathbb{R}^2$ and $\mathbb{R}$ respectively.
\end{proof}
\vspace{0.3cm}
\section{Proofs}
We begin by discussing some results about Floer theoretic properties of Stein cobordisms; these are essentially obtained by rephrasing the results in \cite{Sti} in the language of monopole Floer homology. The key observation here is that the assumption that the Stein filling is \textit{not} negative definite allows us to apply the gluing properties of the Seiberg-Witten invariants; we refer the reader to Chapter $3$ of \cite{KM} for a thorough discussion of the formal picture, and briefly discuss here what we will need.
\par
Suppose we are given a closed oriented $4$-manifold $X$ given as the union $X=X_1\cup_Y X_2$ along a hypersurface $Y$. We will assume that $b_2^+(X_i)\geq 1$ and $Y$ is a rational homology sphere. The latter assumption assures that a spin$^c$ structure $\spin_X$ on $X$ is determined by the two restrictions $\spin_i=\spin_X\lvert_{X_i}$. We will write $\spin_X=\spin_1\cup_Y \spin_2$, and denote by $\spin$ the restriction to $Y$. As $b_2^+(X_i)\geq 1$, we obtain relative invariants
\begin{align*}
\psi_{(X_1,\spin_1)}&\in \HMr_{\bullet}(Y,\spin)\\
\psi_{(X_2,\spin_2)}&\in \HMr_{\bullet}(-Y,\spin)
\end{align*}
where we can identify the last group with $\HMr^{\bullet}(Y,\spin)$ using Poincar\'e duality. The pairing theorem for the Seiberg-Witten invariants $\mathfrak{m}(X,\spin)$ says that
\begin{equation*}
\mathfrak{m}(X,\spin_1\cup_Y\spin_2)=\langle\psi_{(X_1,\spin_1)},\psi_{(X_2,\spin_2)}\rangle \quad \pmod 2
\end{equation*}
where
\begin{equation*}
\langle\cdot,\cdot\rangle: \HMr_{\bullet}(Y,\spin)\otimes \HMr^{\bullet}(Y,\spin)\rightarrow \ztwo
\end{equation*}
is the duality pairing. The relative invariants $\psi_{(X_i,\spin_i)}$ can be interpreted in terms of cobordism maps; for example $\psi_{(X_1,\spin_1)}$ is the image under
\begin{equation*}
\HMf_{\bullet}(X_1\setminus B^4,\spin_1)\rightarrow \HMf_{\bullet}(S^3)\rightarrow \HMf_{\bullet}(Y,\spin)
\end{equation*}
of $1\in\ztwo[[U]]=\HMf_{\bullet}(S^3)$; under our assumption $b_2^+(X_1)\geq1$, the image lies in $\HMr_{\bullet}(Y,\spin)=\mathrm{ker}(p_*)\subset \HMf_{\bullet}(Y,\spin)$. With this in mind, we have the following.
\begin{lemma}\label{isspin0}
Suppose $(W,\omega)$ is a Stein filling of $(Y,\xi)$, and $Y$ is rational homology sphere. Suppose that furthermore $b^+_2(W)>0$, $\s_{\xi}=\s$ is self-conjugate and $\HMr_{\bullet}(Y,\spin)=\ztwo$. Denoting by $\spin_{\omega}$ the canonical spin$^c$ structure of $\omega$, we have:
\begin{itemize}
\item the map induced by $(W\setminus B^4,\spin_{\omega})$ from $\HMf_{\bullet}(S^3)=\ztwo[[U]]$ to $\HMf_{\bullet}(Y,\spin)$ maps $1$ to the generator of $\HMr_{\bullet}(Y,\spin)\subset \HMf_{\bullet}(Y,\spin)$.
\item $\spin_{\omega}$ is self-conjugate.
\end{itemize}
\end{lemma}
\begin{proof}The key input is that any Stein domain $W$ can be embedded in K\"ahler surface $X$ which is minimal of general type \cite{LM}; furthermore, we can assume that $X\setminus W$ is not spin and $b_2^+(X\setminus W)\geq 1$ \cite{Sti}. Now, we know that $X$ has basic classes exactly the canonical classes $\pm K_X$, with Seiberg-Witten invariant $1$ \cite{Mor}. The pairing theorem therefore implies that
\begin{equation*}
\psi_{(W,\spin_{\omega})}\in \HMr_{\bullet}(Y,\spin)=\ztwo
\end{equation*}
is a non-zero element, and the first bullet follows.
\par
Suppose now that $\spin_{\omega}\neq\bar{\spin}_{\omega}$. Denoting by $\spin'$ the canonical spin$^c$ structure on $X\setminus W$, we have
\begin{equation*}
K_X=\spin_{\omega}\cup_Y\spin'\quad\text{and}\quad-K_X=\bar{\spin}_{\omega}\cup_Y\bar{\spin'}.
\end{equation*}
Furthemore, as $X\setminus W$ is not spin, $\spin'\neq \bar{\spin}'$. Now, conjugation acts trivially on $\HMr(Y,\spin)=\ztwo$, so by the pairing theorem $\bar{\spin}_{\omega}\cup_Y\spin'$ and $\spin_{\omega}\\cup_Y\bar{\spin'}$ are also basic classes, which is a contradiction as $\pm K_X$ are the only basic classes.
\end{proof}

\begin{lemma}
Suppose $(W,\omega)$ is a Stein filling of $(Y,\xi)$, and $Y$ is rational homology sphere. Then $b_1(W)=0$.\end{lemma}
\begin{proof}
Topologically, any Stein domain is obtained from $B^4$ by attaching only $1$ and $2$-handles \cite{OzSti}; in particular, the map $\pi_1(Y)\rightarrow \pi_1(W)$ is surjective, and the result follows.
\end{proof}

\begin{proof}[Proof of Theorem \ref{main}]
Let us focus on the case of Type $I$, as the case of Type $II$ is identical. Denote by $\x$ the generator of $\HMr_{\bullet}(Y,\spin)=\ztwo$. By the previous lemma, the spin$^c$ structure $\spin_{\omega}$ on $W$ induced by $\omega$ is self-conjugate; in particular $W$ is spin, hence it has even intersection form. We have the commutative diagram
\begin{center}
\begin{tikzpicture}
\matrix (m) [matrix of math nodes,row sep=3em,column sep=3em,minimum width=2em]
  {
 \HSf_{\bullet}(S^3)  &\HMf_{\bullet}(S^3)\\
  \HSf_{\bullet}(Y,\spin)&\HMf_{\bullet}(Y,\spin)\\};
  \path[-stealth]
  (m-1-1) edge node[above]{$\iota_{S^3}$} (m-1-2)
  (m-1-1) edge node[left]{$\HSf_{\bullet}(W\setminus B^4,\spin_{\omega})$} (m-2-1)
  (m-2-1) edge node[above]{$\iota_{Y}$} (m-2-2)
  (m-1-2) edge node[right]{$\HMf_{\bullet}(W\setminus B^4,\spin_{\omega})$} (m-2-2)
  ;
\end{tikzpicture}
\end{center}
By Lemma \ref{isspin0}, $\HMf_{\bullet}(W\setminus B^4)$ maps $1$ to $\x$; by commutativity of the diagram, we see that $\HSf_{\bullet}(W\setminus B^4)$ maps $1\in\Rin$ to the top degree element of $\HSf_{\bullet}(Y,\spin)$ (cfr. Lemma \ref{observations}). Now, we have also the commutative diagram
\begin{center}
\begin{tikzpicture}
\matrix (m) [matrix of math nodes,row sep=3em,column sep=3em,minimum width=2em]
  {
 \HSf_{\bullet}(S^3)  &\HSb_{\bullet}(S^3)\\
  \HSf_{\bullet}(Y,\spin)&\HSb_{\bullet}(Y,\spin)\\};
  \path[-stealth]
  (m-1-1) edge node[above]{$p_*$} (m-1-2)
  (m-1-1) edge node[left]{$\HSf_{\bullet}(W\setminus B^4,\spin_{\omega})$} (m-2-1)
  (m-2-1) edge node[above]{$p_*$} (m-2-2)
  (m-1-2) edge node[right]{$\HSb_{\bullet}(W\setminus B^4,\spin_{\omega})$} (m-2-2)
  ;
\end{tikzpicture}
\end{center}
As both horizontal maps $p_*$ are injective, we see that $\HSb_{\bullet}(W\setminus B^4)$ is given by multiplication by an element of the form $QV^k$; therefore Proposition \ref{smallb} implies that $b_2^+=1$. Finally, as $\x$ has degree $-2h$ and $1$ has degree $-1$, the map has grading shift $-2h(Y,\spin)+1$, so we also readily obtain the formula for $b_2^-$.
\end{proof}
\vspace{0.3cm}
\begin{proof}[Proof of Corollary \ref{maincor}]
Given a contact three-manifold $(Y,\xi)$, there is a constant $C$ (depending only on $(Y,\xi)$) such that for any Stein filling $W$ we have  $3\sigma(W)+2\chi(W)\geq C$ (Theorem $12.3.14$ in \cite{OzSti}). As in our case $b_1(W)=0$, the latter inequality reads is $5b_2^+(W)\geq b_2^-(W)+C$. In particular an upper bound on $b_2^+(W)$ of fillings, as provided by Theorem \ref{main}, implies an upper bound $b_2^-(W)$, hence on the Euler characteristic.
\end{proof}

\vspace{0.5cm}
\section{Generalizations involving the contact invariant}\label{continv}
As mentioned in the introduction, our main result can be generalized in terms of how the contact invariant $\con(\xi)\in \HMt_{\bullet}(-Y,\spin_{\xi})$ defined in \cite{KMOS} (where we follow the notation of \cite{Ech}) interacts with $\Pin$-monopole Floer homology. In what follows, for a self-conjugate spin$^c$ structure $\spin$ we will work as it is customary with the dual groups
\begin{align*}
\HMt_{\bullet}(-Y,\spin)&\equiv\HMf^{\bullet}(Y,\spin)\\
\HSt_{\bullet}(-Y,\spin)&\equiv\HSf^{\bullet}(Y,\spin).
\end{align*} 
In this setup, recall that there are natural $\ztwo[[V]]$-submodules of $\HSt_{\bullet}(-Y,\spin)$ defined as follows. Consider the map
\begin{equation*}
i_*:\HSb_{\bullet}(-Y,\spin)\rightarrow \HSt_{\bullet}(-Y,\spin).
\end{equation*}
Here the first group can be identified with $\ztwo[[V,V^{-1}][Q]/(Q^3)$, and the map is an isomorphism in degrees high enough and zero in degrees low enough. We then define the $\alpha, \beta$ and $\gamma$-towers to be respectively the images 
\begin{equation*}
i_*(Q^2\cdot\ztwo[[V,V^{-1}]),\quad i_*(Q\cdot\ztwo[[V,V^{-1}]),\quad i_*(\ztwo[[V,V^{-1}]).
\end{equation*}
Notice that the chain involution $\jmath$ of $\Pin$-monopole Floer homology induces an involution at the homology level $\jmath_*$ on $\HMb_{\bullet}(-Y,\spin)$ and $\HMt_{\bullet}(-Y,\spin)$; as $i_*$ is equivariant with respect to these actions, $\jmath_*$ also induces an involution on $\HMr_{\bullet}(-Y,\spin)=\mathrm{ker}p_*=\mathrm{im}j_*$. As it will not cause confusion, in what follows we will also denote the image of $\con(\xi)$ in $\HMr_{\bullet}(-Y,\spin)$ again by $\con(\xi)$.
\begin{thm}\label{thm2}
Consider a contact rational homology sphere $(Y,\xi)$ with $\spin_{\xi}=\spin$ self-conjugate. Assume that the contact invariant $\con(\xi)\in\HMr_{\bullet}(-Y,\spin)$ is invariant under the action of $\jmath_*$, and denote its grading by $d(\xi)$. Consider
\begin{equation*}
\pi_*\con(\xi)\in\HSt_{\bullet}(-Y,\spin).
\end{equation*}
Suppose $(W,\omega)$ is a Stein filling which is \textbf{not} negative definite. Then:
\begin{itemize}
\item if $\pi_*\con(\xi)$ is a non-zero element of the $\beta$-tower, $W$ has even intersection form, $b_2^+=1$ and $b_2^-=5-4d(\xi)$.
\item if $\pi_*\con(\xi)$ is a non-zero element of the $\gamma$-tower, $W$ has even intersection form, $b_2^+=2$ and $b_2^-=10-4d(\xi)$.
\end{itemize}
\end{thm}
This is a genuine generalization of Theorem \ref{main}; this is because, as the discussion below will show, under the assumptions of Theorem $1$ the contact invariant $\con(\xi)$ is the generator of $\HMr_{\bullet}(-Y,\spin)=\ztwo$. This is of course invariant under the action of $\jmath_*$, and furthermore when $(Y,\spin)$ has Type $I$ (resp. Type $II$), $\pi_*$ maps it to the $\beta$ (resp. $\gamma$) tower.
\begin{remark}
We have omitted in the statement the case $\pi_*\con(\xi)$ belongs to the $\alpha$-tower; this is because, as the proof below would show, $W$ would be negative definite.
\end{remark}
As a consequence, we have the following finiteness result.
\begin{cor}
Suppose that a contact rational homology sphere $(Y,\xi)$ with $\spin_{\xi}=\spin$ self-conjugate satisfies the following:
\begin{itemize}
\item $\con(\xi)$ is invariant under the action of $\jmath_*$;
\item $\pi_*\con(\xi)$ is a non-zero element in $i_*\left(\HSb_{\bullet}(-Y,\spin)\right)$.
\end{itemize}
Then the set of Euler characteristics of Stein fillings of $(Y,\xi)$ is finite.
\end{cor}

As a simple example to which our result applies, we look for $n\geq1$ at the Stein fillable contact structure $\xi_n$ on the Brieskorn homology sphere $Y_n=\Sigma(2,2n+1,4n+3)$ constructed in \cite{Kar} by Legendrian surgery on a certain stabilized Legendrian $(2,2n+1)$-torus knot. Provided $n>1$, this manifold has $\HMr_{\bullet}$ of rank strictly larger than $1$. Let us focus on the case $n=4k+1$ for $k\geq1$. Following \cite{Kar}, we see that
for some $\ztwo[[U]]$-module $J$,
\begin{equation*}
\HMt_{\bullet}(-Y_{4k+1})=\mathcal{T}^+_0\oplus (\ztwo[[U]]/U^{(2k+1)})_0 \oplus J^{\oplus 2}.
\end{equation*}
The computation of the Heegaard Floer contact invariant provided in \cite{Kar}, together with its identification with the one in monopole Floer homology (see \cite{Tau}, \cite{CGH} and subsequent papers), implies that $\con(\xi_{4k+1})$ is either $(0, U^{2k}, 0,0)$ or $(1, U^{2k},0,0)$ (i.e. either the bottom element of the second summand, or its sum with the bottom of the tower).
\par
In this example, both the involution $\jmath_*$ and the map $\pi_*$ are explicitly computable (see \cite{Sto}, \cite{Dai} for a more general discussion of Seifert-fibered spaces). In particular, the action of $\jmath_*$ fixes the first two summands and switches the two copies of $J$, and $\pi_*$ maps the contact invariant $\con(\xi_{4k+1})$ to the bottom of the $\gamma$-tower. Below we depict graphically (omitting the $J^{\oplus 2}$ summand, which is irrelevant for our purposes) the Gysin exact sequence in the case in which $k=1$.
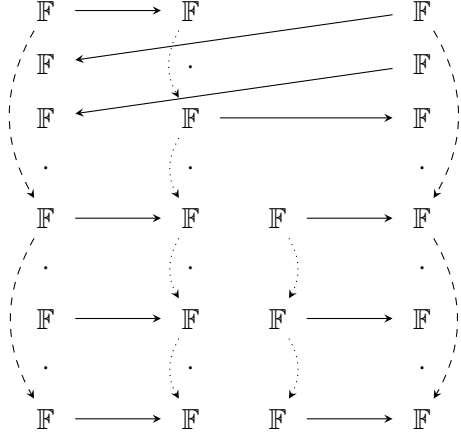
\begin{center}
\begin{tikzpicture}
\matrix (m) [matrix of math nodes,row sep=0.5em,column sep=1em,minimum width=2em]
  {
 \ztwo && &\ztwo&&&&\ztwo\\
  \ztwo&&&\cdot&&&&\ztwo\\
  \ztwo&&&\ztwo&&&&\ztwo\\
  \cdot&&&\cdot&&&&\cdot\\
  \ztwo&&&\ztwo&\ztwo&&&\ztwo\\
  \cdot&&&\cdot&&&&\cdot\\
  \ztwo&&&\ztwo&\ztwo&&&\ztwo\\
  \cdot&&&\cdot&&&&\cdot\\
  \ztwo&&&\ztwo&\ztwo&&&\ztwo\\};
  \path[-stealth]
  (m-1-1) edge node{} (m-1-4) 
   (m-5-1) edge node{} (m-5-4) 
    (m-7-1) edge node{} (m-7-4) 
     (m-9-1) edge node{} (m-9-4) 
     (m-3-4) edge node{} (m-3-8) 
      (m-5-5) edge node{} (m-5-8) 
      (m-7-5) edge node{} (m-7-8) 
      (m-9-5) edge node{} (m-9-8)
      (m-1-8) edge node{} (m-2-1)
      (m-2-8) edge node{} (m-3-1)
      (m-1-4) edge[dotted, bend right] node{} (m-3-4)
      (m-3-4) edge[dotted, bend right] node{} (m-5-4)
      (m-5-4) edge[dotted, bend right] node{} (m-7-4)
      (m-7-4) edge[dotted, bend right] node{} (m-9-4)
      (m-5-5) edge[dotted, bend left] node{} (m-7-5)
      (m-7-5) edge[dotted, bend left] node{} (m-9-5)
      (m-1-1) edge[dashed, bend right] node{} (m-5-1)
       (m-5-1) edge[dashed, bend right] node{} (m-9-1)
       (m-1-8) edge[dashed, bend left] node{} (m-5-8)
       (m-5-8) edge[dashed, bend left] node{} (m-9-8)

   ;
\end{tikzpicture}
\end{center}
In the picture, the relevant part of $\HMt_{\bullet}$ is represented by the two middle columns, with the dotted arrows denoting the $U$-action. The dashed arrows represent the $V$-action on the $\gamma$-tower of $\HSt_{\bullet}$.
Our result then implies that any Stein filling of $(Y_{4k+1},\xi_{4k+1})$ which is not negative definite is spin with $b_2^+=2$ and $b_2^-=10$, hence has intersection form $2H\oplus E_8$.
\\
\par

As in the proof Theorem \ref{main}, an important step is to show the following.
\begin{lemma}\label{isspin}
Suppose $\con(\xi)\in\HMr_{\bullet}(-Y,\spin)$ is $\jmath_*$ invariant. Then for any Stein filling $(W,\omega)$ of $(Y,\xi)$ with $b_2^+>0$, $\spin_{\omega}$ is self-conjugate.
\end{lemma}
Given this lemma (which we will prove later) only missing ingredient is the following recently proved naturality for the contact invariant in monopole Floer homology.
\begin{thm}[Theorem $1$ of \cite{Ech}]
Let $(W,\omega)$ be a strong symplectic cobordism between $(Y,\xi)$ and $(Y',\xi')$, and let $\spin_{\omega}$ the corresponding spin$^c$ structure on $W$. We have
\begin{equation*}
\con(\xi)=\HMt_{\bullet}(W^{\dagger},\spin_{\omega})\con(\xi').
\end{equation*}
where $W^{\dagger}$ denotes the reversed cobordism from $-Y'$ to $-Y$,
\end{thm}
\begin{proof}[Proof of Theorem \ref{thm2}]
Recall that the contact invariant of $(S^3,\xi_{std})$ is the bottom of the tower of $\HMt_{\bullet}(-S^3,\xi_{std})$. By removing a standard ball around a point, we get a strong symplectic cobordism between $(S^3,\xi_{std})$ and $(Y,\xi)$, so that $\HMt_{\bullet}((W\setminus B^4)^{\dagger},\spin_{\omega})$ maps $\con(\xi)$ to the bottom of the tower of $\HMt_{\bullet}(-S^3)$. Suppose $W$ is not negative definite, so that by Lemma \ref{isspin} the induced spin$^c$ structure $\spin_{\omega}$ is self-conjugate. Then we get a commutative diagram of the form
\begin{center}
\begin{tikzpicture}
\matrix (m) [matrix of math nodes,row sep=3em,column sep=3em,minimum width=2em]
  {
 \HMt_{\bullet}(-Y,\spin)  &\HSt_{\bullet}(-Y,\spin)\\
  \HMt_{\bullet}(-S^3)&\HSt_{\bullet}(-S^3)\\};
  \path[-stealth]
  (m-1-1) edge node[above]{$\pi_{-Y}$} (m-1-2)
  (m-1-1) edge node[left]{$\HMt_{\bullet}((W\setminus B^4)^{\dagger},\spin_{\omega})$} (m-2-1)
  (m-2-1) edge node[above]{$\pi_{-S^3}$} (m-2-2)
  (m-1-2) edge node[right]{$\HSt_{\bullet}((W\setminus B^4)^{\dagger},\spin_{\omega})$} (m-2-2)
  ;
\end{tikzpicture}
\end{center}
Here $\pi_{-S^3}$ sends the bottom of the tower to the bottom of the $\alpha$-tower. Focusing now on the first bullet of the theorem, this implies that $\HSt_{\bullet}((W\setminus B^4)^{\dagger},\spin_{\omega})$ sends an element of the $\beta$-tower (i.e. $\pi_*\con(\xi)$) to the bottom of the $\alpha$-tower; as both these elements are in the image of the respective map $i_*$, looking at the commutative diagram
\begin{center}
\begin{tikzpicture}
\matrix (m) [matrix of math nodes,row sep=3em,column sep=3em,minimum width=2em]
  {
 \HSb_{\bullet}(-Y,\spin)  &\HSt_{\bullet}(-Y,\spin)\\
  \HSb_{\bullet}(-S^3)&\HSt_{\bullet}(-S^3)\\};
  \path[-stealth]
  (m-1-1) edge node[above]{$i_*$} (m-1-2)
  (m-1-1) edge node[left]{$\HSb_{\bullet}((W\setminus B^4)^{\dagger},\spin_{\omega})$} (m-2-1)
  (m-2-1) edge node[above]{$i_*$} (m-2-2)
  (m-1-2) edge node[right]{$\HSt_{\bullet}((W\setminus B^4)^{\dagger},\spin_{\omega})$} (m-2-2)
  ;
\end{tikzpicture}
\end{center}
we obtain from Lemma \ref{smallb} that $b_2^+=1$, and furthermore $b_2^-$ can be understood in terms of the grading shift.
\end{proof}

We conclude by proving Lemma \ref{isspin}, which is essentially a refinement of Lemma \ref{isspin0}.

\begin{proof}[Proof of Lemma \ref{isspin}]
Consider again the embedding $W\hookrightarrow X$ into a minimal K\"ahler surface of general type, with $V=X\setminus W$ not spin and $b_2^+(V)>0$. This embedding is K\"ahler, so in particular $V$ is a strong \textit{concave} filling of $(Y,\xi)$. This implies that $\con(\xi)$ is in fact the relative invariant of $(V^{\dagger},\spin_{\omega})$, i.e. 
\begin{equation*}
\con(\xi)=\HMa_{\bullet}(V^{\dagger}\setminus B^4,\spin_{\omega})(1)\in \HMr_{\bullet}(-Y,\spin).
\end{equation*}
While this is not explicitly stated in \cite{Ech}, the gluing results in the \textit{dilating the cone} approach of discussed in his Section $2$ applies directly to show a correspondence between moduli spaces on
\begin{equation*}
(\mathbb{R}^+\times\{-Y\})\cup([1,\infty)\times Y),
\end{equation*}
which are used to define the contact invariant $\con(\xi)$, and on
\begin{equation*}
(\mathbb{R}^+\times\{-Y\})\cup V
\end{equation*}
where in the latter case we use a symplectic form for which the cone is dilated enough. In the latter case all solutions are irreducible, so they define a chain representing the relative invariant of the cobordism $V^{\dagger}$.
\par
The lemma now follows as Lemma \ref{isspin0}, using the fact that $X$ has exactly two basic classes $\pm K_X$, the Seiberg-Witten invariants relative of to $\spin_1\cup_Y\spin_2$ is the coefficient of $1\in \HMt_{\bullet}(S^3)$ in
\begin{equation*}
\HMt_{\bullet}((W\setminus B^4)^{\dagger},\spin_1)) \circ\HMa_{\bullet}((V\setminus B^4)^{\dagger},\spin_2)(1)
\end{equation*}
and the fact that $\con(\xi)$ is invariant under the conjugation action (so that we also have $\con(\xi)=\HMa_{\bullet}(V^{\dagger}\setminus B^4,\bar{\spin}_{\omega})(1)$).
\end{proof}
\begin{remark}
A similiar identification of the contact invariant as a relative invariant is proved in the setting of Heegaard Floer homology in \cite{Pla}; there the author discusses the specific case in which $V$ is explicitly constructed in terms of the open book decomposition of $(Y,\xi)$.
\end{remark}

\vspace{0.5cm}

\bibliographystyle{alpha}
\bibliography{biblio}

\end{document}